\documentclass[reqno]{amsart}
\usepackage{amsmath, amsthm, amsfonts, amscd, graphicx}
\usepackage{pbox} 

    \theoremstyle{plain}
    \newtheorem{Thm}{Theorem}[section]
    
    \newtheorem{Lemma}[Thm]{Lemma}
    \newtheorem*{Lemma*}{Lemma}

    \theoremstyle{definition}
    \newtheorem{Def}[Thm]{Definition}
    \newtheorem*{Def*}{Definition}
    \newtheorem{Example}[Thm]{Example}

    \theoremstyle{remark}
    
    \newtheorem*{Remark*}{Remark}

    \numberwithin{equation}{section}

    \newcommand{\field}[1]{\mathbb{#1}}
    
    \newcommand{\Z}{\field{Z}}

    \newcommand{\C}{\field{C}}
    \DeclareMathSymbol{\fieldk}{\mathalpha}{AMSb}{"7C} 

    \renewcommand{\H}{\mathbb{H}}
    \newcommand{\D}{\mathbb{D}}
        
    \newcommand{\abs}[1]{\lvert#1\rvert}

    \newcommand{\floor}[1]{\left\lfloor#1\right\rfloor}
    \DeclareMathOperator{\Tr}{Tr}

    \DeclareMathSymbol{\normal}{\mathord}{AMSa}{"43}
    \DeclareMathOperator{\Log}{Log}
    \DeclareMathOperator{\Exp}{Exp}

    \newcommand{\define}[1]{\emph{#1}}      

    \newcommand{\ltwo}{l^{2}}
    \newcommand{\Ltwo}{L^{2}}
    
    \newcommand{\group}[1]{\langle {#1} \rangle}
    
        \newcommand{\len}{\ell}   
        \newcommand{\adj}{A}

         \newcommand{\ellipticK}{\mathbf{K}}
	\newcommand{\detDeltau}{\det{}_\pi \Delta_u}
	
\begin{document}

%
%
\title{The {I}hara zeta function of the infinite grid}

\author{Bryan Clair}
\address{Saint Louis University,
         220 N. Grand Avenue
		 St. Louis, MO 63108}
\email{bryan@slu.edu}
\date{\today}


\begin{abstract}
The infinite grid is the Cayley graph of $\mathbb{Z} \times \mathbb{Z}$ with the usual generators.  In this paper, the Ihara zeta function for the infinite grid
is computed using elliptic integrals and theta functions.
The zeta function of the grid extends to an analytic, multivalued function which satisfies a functional equation.
The set of singularities in its domain is finite.

The grid zeta function is the first computed example which is non-elementary,
and which takes infinitely many values at each point of its domain.  It is also the limiting value of the normalized sequence of
Ihara zeta functions for square grid graphs and torus graphs.
\end{abstract}

\maketitle

\section{Introduction}
Graph theory zeta functions first appeared in the 1960's, introduced by Ihara and later refined by Serre, Hashimoto, and Bass. 
Work of Stark and Terras in the late 1990's popularized the subject.  Today, the definitive reference is Terras' book ~\cite{terras:book}.

Like the Riemann zeta function, the zeta function of a graph is an Euler product taken over primes.  In the case of graphs, the primes are certain closed paths, analogous to the geodesics used as primes in Selberg's zeta function for a hyperbolic surface.  For these zeta functions, the Euler product only converges on a small open set, and analytic continuation to the plane is performed by a functional equation.

For zeta functions of finite graphs, the Ihara determinant formula connects the poles of the graph zeta function to the spectrum of the adjacency operator on the graph.  Moreover, the formula shows that the zeta function of a graph is actually the reciprocal of a polynomial, thereby providing the analytic continuation of the zeta function to a meromorphic function on the plane.  Functional equations follow as a straightforward byproduct.
The poles of the zeta function are closely related to the distribution of prime cycles in the graph, and one can prove a ``graph prime number theorem''.

Generalizations to infinite graphs were introduced by the author and Mokhtari-Sharghi in ~\cite{cms:zeta}, and independently by Grigorchuk and Zuk~\cite{grigzuk}.  The paper by Guido, Isola, and Lapidus~\cite{gil:zeta}  is the best introduction to the subject.
There are a few settings in which an infinite graph has a zeta function, but the central objects of study are periodic graphs with finite quotient.    Such graphs appear naturally as Cayley graphs of finitely generated discrete groups.

The Euler product definition of the zeta function for infinite graphs mirrors closely the definition for finite graphs, and there is an analog of the Ihara formula giving the zeta function in terms of the determinant of a Laplacian operator.  With finite graphs, this operator acts on the finite dimensional space of functions on vertices, and so the determinant is a polynomial.  For infinite graphs it naturally acts on the Hilbert space of $L^2$ functions on vertices, may have continuous spectrum, and defining the determinant is no longer straightforward.

The presence of continuous spectrum in infinite periodic graphs may prevent the zeta function from having an analytic continuation to the plane.
The fundamental problem in the study of zeta functions of infinite graphs is to provide this continuation.  Without it, the location of zeta function singularities is undefined, and one cannot formulate a Riemann hypothesis or prove an analog of the prime number theorem.

At the time of this writing there is no general theorem providing an analytic continuation, although the recent paper~\cite{gi:duals} gives some encouraging positive results. On the other hand, all known examples do have analytic continuations (albeit as multivalued functions) with finitely many isolated singularities.  Unfortunately, there are very few examples where the zeta function can be computed.  This paper aside, all explicit examples are contained in~\cite{clair:zetaz} and~\cite{grigzuk}.  Those that are computable turn out to be algebraic functions, and so can be considered as functions on compact Riemann surfaces.

This paper is concerned with a single example, the Cayley graph of $\Z \times \Z$ with the usual generators,
henceforth known as ``the grid''.  The bulk of the work is in Sections~\ref{sec:integration} and~\ref{sec:uniformization},
which explicitly compute the zeta function for the grid in terms of elliptic theta functions.

Theorem~\ref{thm:extends} shows that the grid zeta function does extend analytically to a multivalued function with a finite set of isolated singularities.  Theorem~\ref{thm:notelementary} shows that the grid zeta function is a non-elementary function, and so is the first non-algebraic example.  The grid also provides the first example whose zeta function requires infinitely many values -- the Riemann surface on which it lives is noncompact and is an infinite sheeted cover of the plane.  A functional equation for the grid zeta function follows easily from the explicit formula, and is proved in Theorem~\ref{thm:duality}.

For completeness, Section~\ref{sec:combinatoric} derives a combinatorial formula for the power series coefficients of the grid zeta function,
a result first presented without proof in~\cite{cms:zeta}.

\section{Basic definitions and known examples}
\begin{figure}
\includegraphics{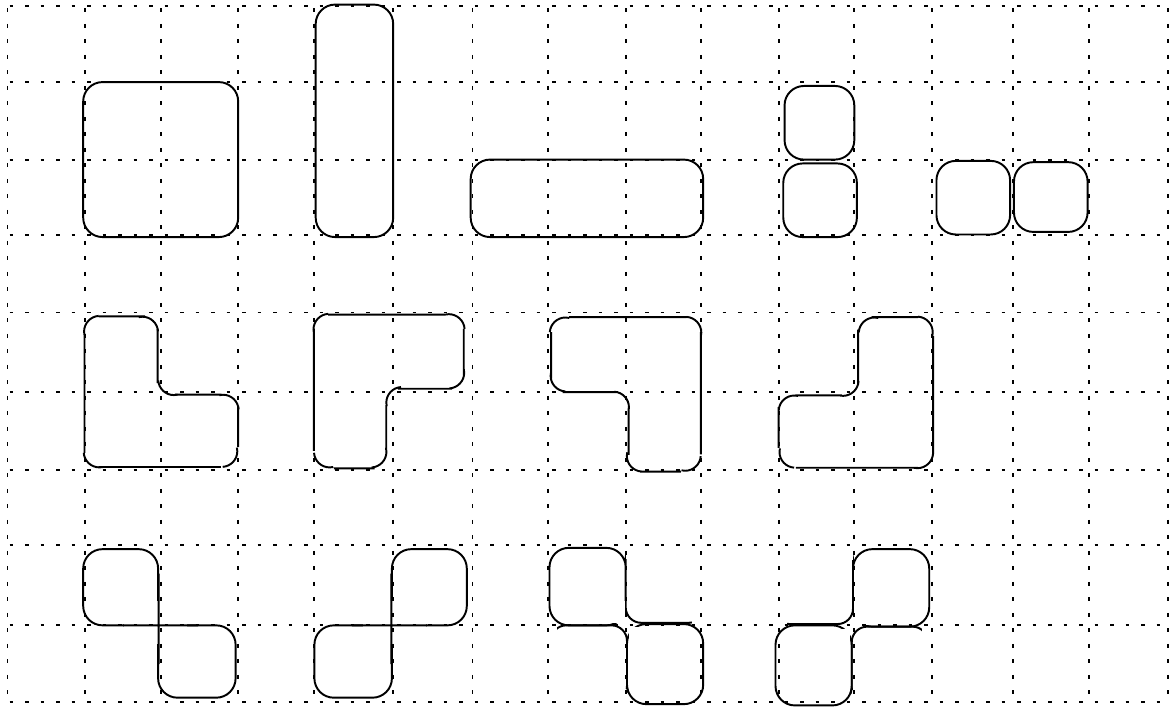}
\caption{Primitive paths of length 8}
\label{fig:8paths}
\end{figure}

A periodic graph is a graph $X$ equipped with an action by a discrete group $\pi$ such that $X/\pi$ is a finite graph, and
$X \to X/\pi$ is a regular covering.  

Let $\mathcal{P}$ denote the set of free homotopy classes of primitive closed paths in $X$.  Equivalently, $\mathcal{P}$ consists 
of primitive closed paths without backtracking, up to cyclic equivalence.  Here, primitive means that the path is not simply a shorter path repeated some number of times.  For $\gamma\in \mathcal{P}$, $\len(\gamma)$ is the length of $\gamma$.
As an example, all primitive paths of length 8 on the grid are shown in Figure~\ref{fig:8paths}, up to translation and reversal of orientation.

For simplicity, we restrict to the case where $\pi$ is torsion free, so that $\pi$ acts freely on $\mathcal{P}$.
\begin{Def}
The \define{$\Ltwo$-zeta function} of $X$ is the infinite product
\begin{equation}\label{eq:zprod}
   Z_{\pi}(u) = \prod_{\gamma \in \mathcal{P}/\pi}
                   \left(1-u^{\len(\gamma)}\right)^{-1},
\end{equation}
which converges for small $u$.
\end{Def}

The $L^2$-Euler characteristic of $X$ is $\chi^{(2)}(X) = v-e$, where $v = \# VX/\pi$ and $e = \#EX/\pi$ are the number of vertices and edges
of $X/\pi$.  The adjacency operator on $X$ is $\adj : \ltwo(VX) \to \ltwo(VX)$. The Ihara Laplacian is
\begin{equation}
\Delta_u = I - \adj u + qu^2
\end{equation}
and Ihara's formula becomes
\begin{equation}\label{eq:zetadef}
   Z_{\pi}(u)^{-1} = (1-u^{2})^{-\chi^{(2)}(X)}\detDeltau.
\end{equation}
Here, the determinant $\detDeltau$ is defined via formal power
series as $(\Exp \circ \Tr_{\pi} \circ \Log) \Delta_u$, where $\Tr_{\pi}$ is the trace
on the von Neumann algebra of $\pi$ invariant operators on $ \ltwo(VX)$.
See~\cite{gil:zeta} and \cite{cms:zeta} for details and proofs.
The only fact about von Neumann trace needed in this paper is that an element of the group ring
$M = \sum_{\gamma \in \pi} m_\gamma \in \C[\pi]$ acting on $\ltwo(\pi)$ is in the von Neumann algebra and
has trace $\Tr_\pi M = m_e$.

The formula~\eqref{eq:zetadef} analytically continues the zeta function to a set $\Omega$ containing $u = 0$.
$\Omega$ is the finite component of $\C - D$, where $D$ is the union of the segments $[-1,-\frac{1}{q}] \cup [\frac{1}{q},1]$ and the circle at the origin of radius $1/\sqrt{q}$.  Figure~\ref{fig:Y} shows these sets in the case $q = 3$.

\begin{Example}\label{ex:path}
The infinite path $X$ is the Cayley graph of $\Z = \group{a}$.  Here $X/\Z$ is a single vertex with a single loop edge, and so $\chi^{(2)}(X) = 0$.  The operator $\Delta_u$ acts on $\ltwo(\Z)$, and $\Delta_u = 1 - (a + a^{-1})u + u^2$ in the group ring $\C[\group{a}]$.
Fourier transform identifies $\ltwo(\Z)$ with $\Ltwo(S^1)$, and $a$ with $e^{i\theta}$ acting by multiplication, and so one can compute the von Neumann trace by integration over $S^1$.  Then for $-1 < u < 1$,
\begin{align}
\log \detDeltau = \Tr_\pi \Log \Delta_u = \int_{S^1} \log (1 - 2u \cos(\theta) + u^2) d\theta
\end{align}

On the other hand, there are no loops in $X$, so $Z_\pi(u) = 1$ trivially from the Euler product definition.  Applying~\eqref{eq:zetadef},
$1 = \detDeltau$, so that
\begin{align}
   0 = \log \detDeltau &= \int_{S^1} \log (1 - 2u \cos(\theta) + u^2) d\theta\\
   & = \log(1+u^2) + \int_{S^1} \log\left(1 - \frac{2u}{1+u^2}\cos\theta\right) d\theta.
\end{align}
Setting $z = \frac{2u}{1+u^2}$ and rearranging,
\begin{equation}\label{eq:zint}
\int_{S_1} \log(1 - z \cos \theta) d\theta = \log\left(\frac{1}{2}\bigl(1 + \sqrt{1-z^2}\ \bigr)\right),
\end{equation}
which holds for any $z \in (-1,1)$.  This integral will be needed later, and could be derived using some clever calculus rather than the Ihara determinant formula, if desired.
\end{Example}

\begin{Example}\label{ex:tree}
Suppose $\pi = F_n$ is the free group on $n$ generators, generalizing Example~\ref{ex:path}. The Cayley graph $X$ of $F_n$ is an infinite $2n$-regular tree, and because it has no loops, $Z_\pi(u) = 1$. 
\end{Example}

\begin{Example}
The case $\pi = \Z$.
Let $X$ be any graph with a $\Z$-action, so that $X \to X/\Z$ is a regular covering of the finite graph $X/\Z$.
Then the main theorem of~\cite{clair:zetaz} shows that
\[ Z_\pi(u)^{-1} = p(u)\prod_{i=1}^{N}\left(r_i \pm \sqrt{r_i^2 - 1}\right) \]
where $p$ is a polynomial and $r_i(u)$ are the roots of a degree $N$ polynomial with coefficients which are integer polynomials in $u$.
In particular, $Z_\pi(u)$ is multivalued and an algebraic function of $u$.  It lives on a compact Riemann surface covering $\C$, with
isolated singularities.
\end{Example}

\begin{Example}
Grigorchuk and Zuk~\cite{grigzuk} provide a number of examples of $\Ltwo$ zeta functions, but most are left as integrals or infinite products.  Their examples are all for Cayley or Schrier graphs of groups generated by automata.

The first examples are for groups of intermediate growth where the spectra of $\Delta_u$ are absolutely continuous.  They involve integrals of the form
\[ \int \frac{\log\bigl(\alpha(u) + \beta(u) x\bigr)}{\sqrt{\gamma(x)}} dx \]
where $\alpha$,$\beta$, and $\gamma$ are polynomials.  In the cases where $\gamma(x)$ is quadratic, a trigonometric substitution reduces the integrals to~\eqref{eq:zint}, and the zeta functions are algebraic.  When $\gamma(x)$ is of higher degree, the integrals are apparently intractable and the analytic structure of the zeta function is unclear.

In their second collection of examples, the adjacency operators have totally discontinuous spectrum.  They are able to express these zeta functions generally as infinite products of the form
\[  \alpha(u)\prod_{n=1}^{\infty} \beta_n(u)^{\rho(n)} \]
where $\alpha$ is rational, $\beta_n$ are a family of polynomials of fixed degree, and the exponents $\rho(n)$ approach zero, exponentially in $n$. 
These examples seem like likely candidates for zeta functions which may not be continued past the set $D$, since the spectra accumulate.  However, the presence of the roots $\rho(n)$ prevent any clear analysis.
\end{Example}

%
%
\section{Computation of the grid zeta function}\label{sec:integration}
From this point forward, only the grid is considered.

\subsection{The grid and an integral over the torus}

Let $\pi = \Z \times \Z = \group{a} \times \group{b}$, and
let $X$ be the Cayley graph of $\pi$ with generators $\{a,b\}$, ``the grid''. 
$X$ is regular of degree $q + 1 =  4$, and $X/\pi$ is a graph
with one vertex and two loops, so that $\chi^{(2)}(X) = -1$.
$Z_\pi$ will refer to the $\Ltwo$ zeta function of the grid.
Since any closed loop on the grid has an even number of edges, $Z_{\pi}(u)$ is an even function.

The adjacency operator $A$ for $X$ is  $a + a^{-1} + b + b^{-1} \in \C[\pi]$, and
\[ \Delta_u = 1 - (a + a^{-1} + b + b^{-1})u + 3u^2 \]
acting on $\ltwo(\Z \times \Z)$.

The starting point to compute $Z_{\pi}(u)$ is to apply a Fourier transform and
change the problem to the computation of an integral over the torus.
Fourier transform gives $\ltwo(\Z \times \Z)  \cong \Ltwo(S^1 \times S^1)$.  Using $(s,t)$ as coordinates on $S^1 \times S^1$, the Fourier transform of $\Delta_u$ acts by multiplication on $ \Ltwo(S^1 \times S^1)$ by
\begin{equation}
  \hat{\Delta}_u = 1 + 3u^2 - 2u\cos(t) - 2u\cos(s).
\end{equation}
We then have:
\begin{align}
\log \detDeltau &= \Tr_\pi \Log \Delta_u\\
 &= \int \int_{S^1\times S^1} \log\bigl(1+ 3u^2 - 2u\cos(t) - 2u\cos(s)\bigr)\ ds dt
 \label{eq:torusintegral}
\end{align}
This integral is an analytic function of $u$ for $u$ near zero.  More precisely, define the 
sets
\begin{align}\label{eq:badset}
D &= \left\{u\ \Big|\ \abs{u} = 1/\sqrt{3}\right\} \cup [-1,-1/3] \cup [1/3,1]\\
\Omega &= \text{the component of $\C - D$ containing zero}, \label{eq:zdomain}
\end{align}
both shown in Figure~\ref{fig:Y}.

\begin{figure}
\includegraphics{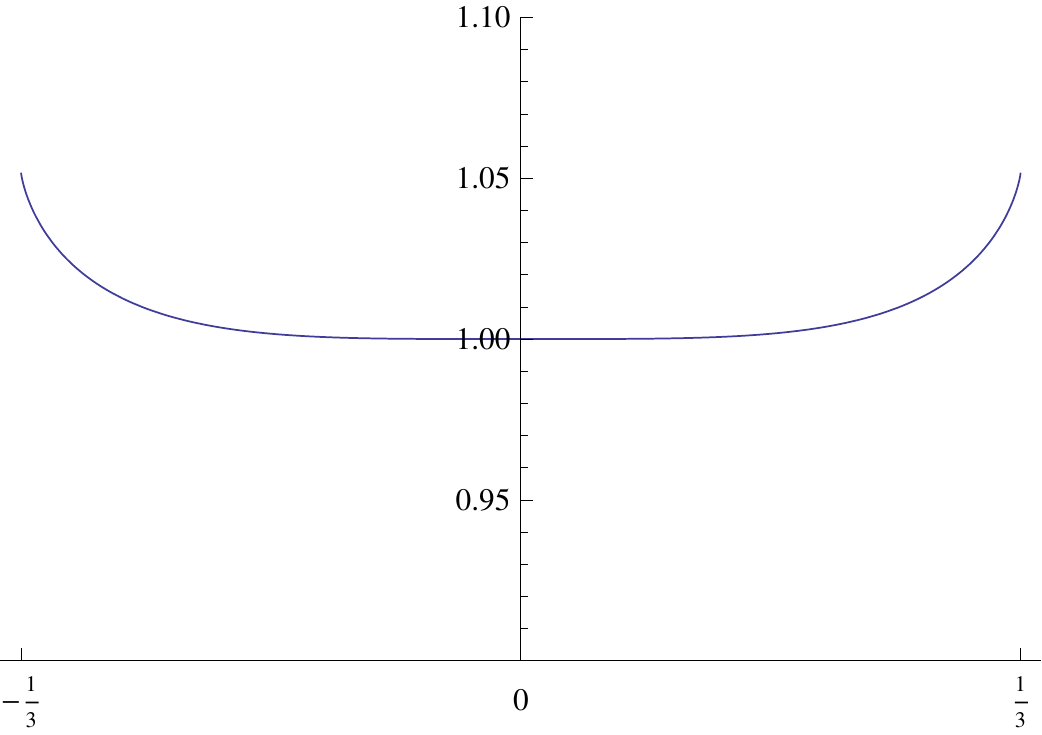}
\caption{$Z_\pi(u)$ for $u \in (-1/3,1/3)$.}
\label{fig:realzeta}
\end{figure}

The integrand $1+ 3u^2 - 2u\cos(t) - 2u\cos(s)$ avoids the set $(-\infty,0]$ as long as $u \in \Omega$,
which shows that~\eqref{eq:torusintegral} and therefore $Z_{\pi}(u)$ are analytic on $\Omega$.
At this point, numeric integration of~\eqref{eq:torusintegral} can produce Figure~\ref{fig:realzeta}.

\subsection{Evaluation of the torus integral}
This section computes the Ihara zeta function for the grid by evaluating the integral formula~\eqref{eq:torusintegral}
for $\log \detDeltau$.   It is worth mentioning that in~\cite{cjk:tori}, the authors tackled the integral~\eqref{eq:torusintegral} with $u=1$ in the context of discrete tori, using a heat equation approach.  Their methods are nearly applicable here, but a series in their work that converges to the Catalan constant diverges when $u <1$, preventing a straightforward generalization.

Instead, we take a totally different approach with roots in statistical mechanics.
An integral similar to~\eqref{eq:torusintegral} appears when solving the Ising model on
Onsager's lattice (the grid) with no magnetic field.  The key ideas in this section originate in these physical studies, dating from the 1940's.
See, for example, \cite[Ch. 5]{mccoywu}.

\begin{Thm}\label{thm:logderiv}
For all $u \in \Omega-\{0\}$,
\begin{equation}
\frac{d}{du}\log \detDeltau = u^{-1} - 2\pi^{-1}\ellipticK(k)k^{-1}\frac{dk}{du} \label{eq:logderiv},
\end{equation}
where $k = \frac{4u}{1+3u^2}$ and $\ellipticK$ is the complete elliptic integral of the first kind.
\end{Thm}
\begin{proof}
Assume $0 < u < \frac{1}{3}$.  Then
\begin{align}
\log \detDeltau &= \Tr_\pi \Log \Delta_u\\
 &= \int \int_{S^1\times S^1} \log(1+ 3u^2 - 2u\cos(t) - 2u\cos(s)) ds dt \\
 &= \log(1+3u^2) + (2\pi)^{-2}\int_0^{2\pi}\hspace{-0.8em}\int_0^{2\pi}
        \log\left[1 - \frac{1}{2}k(\cos(t) + \cos(s))\right] ds dt\\
 &= \log(1+3u^2) + \pi^{-2}\int_0^{\pi}\hspace{-0.6em}\int_0^{\pi}
        \log\left[1 - \frac{1}{2}k(\cos(t) + \cos(s))\right] ds dt.
\end{align}

Now make the change of variables $\tau = (s + t)/2$, $\omega = (s-t)/2$, so that $\cos(s) + \cos(t) = 2\cos(\tau)\cos(\omega)$.  The Jacobian introduces a factor of two.  The region of integration becomes the square with corners $(0,0)$, $(\pi/2,\pm\pi2)$, and $(\pi,0)$ in the $(\tau,\omega)$ plane, but due to the symmetry of the integrand the same result comes from integrating over $(\tau,\omega) \in [0,\pi] \times [0,\pi/2]$.  That is,
\begin{equation}
\log \detDeltau
  = \log(1+3u^2) + 2\pi^{-2}\int_0^{\pi/2}\hspace{-0.6em}\int_0^{\pi}
        \log\left[1 - k\cos(\tau)\cos(\omega)\right] d\tau d\omega.
\end{equation}

Using~\eqref{eq:zint}, the integral over $\tau$ can be performed, resulting in:
\begin{align}
\log \detDeltau
  &= \log(1+3u^2) + 2\pi^{-1}\int_0^{\pi/2}
        \log\frac{1}{2}\left[1+\sqrt{1 - k^2\cos^2(\omega)}\right] d\omega \\
  &= \log(\frac{1+3u^2}{2}) + 2\pi^{-1}\int_0^{\pi/2}
        \log\left[1+\sqrt{1 - k^2\sin^2(\omega)}\right] d\omega \label{eq:integralform}
\end{align}

Now take the derivative of both sides, and 
put $\sigma  = \sqrt{1-k^2\sin^2(\omega)}$.
\begin{align}
\frac{d}{du}\log\detDeltau
  &= \frac{6u}{1+3u^2} - 2\pi^{-1}k\frac{dk}{du}\int_0^{\pi/2}
        \frac{\sin^2(\omega)}{\sigma(1+\sigma)}
        d\omega\\
   &= \frac{6u}{1+3u^2} - 2\pi^{-1}k\frac{dk}{du}\int_0^{\pi/2}
        \frac{1-\sigma}{k^2\sigma}
        d\omega\\
    &=  \frac{6u}{1+3u^2} -  2\pi^{-1}k^{-1}\frac{dk}{du}\int_0^{\pi/2}
        \frac{1}{\sigma} - 1
        d\omega\\
    &=  \frac{6u}{1+3u^2} +k^{-1}\frac{dk}{du}\left(1 - 2\pi^{-1}\int_0^{\pi/2}
        \frac{d\omega}{ \sqrt{1-k^2\sin^2(\omega)}}\right)
\end{align}

The integral is exactly the complete elliptic integral of the first kind:
\begin{equation}\label{eq:ellipticK}
   \ellipticK(k) = \int_0^{\pi/2}
        \frac{d\omega}{ \sqrt{1-k^2\sin^2(\omega)}}.
\end{equation}

Using $k^{-1}\frac{dk}{du} = u^{-1}(\frac{1-3u^2}{1+3u^2})$,
\begin{align}
\frac{d}{du}\log \detDeltau
    &=  \frac{6u}{1+3u^2} +u^{-1}(\frac{1-3u^2}{1+3u^2})
    - 2\pi^{-1}\ellipticK(k)k^{-1}\frac{dk}{du}\\
    &= u^{-1} - 2\pi^{-1}\ellipticK(k)k^{-1}\frac{dk}{du}. \label{eq:ldsource}
\end{align}

The mapping $u \to k$ is a homeomorphism between the set $\Omega$ where $Z_\pi(u)$ is analytic and
the set $\C - (-\infty,-1] \cup [1,\infty)$ where the principal branch of $\ellipticK$ is analytic.
Since~\eqref{eq:ldsource} holds for $u \in (0,\frac{1}{3})$, it must remain true for $u \in \Omega - \{0\}$.
\end{proof}

Recall that an elementary function is constructed by repeatedly taking algebraic operations, logarithms, and exponentials of a variable.

\begin{Thm}\label{thm:notelementary}
The zeta function $Z_{\pi}(u)$ is not an elementary function.
\end{Thm}
\begin{proof}
By~\eqref{eq:zetadef}, the logarithmic derivative of $Z_{\pi}(u)$ can be written in terms of the logarithmic derivative of $\det{}_\pi \Delta_u$.  Then, applying~\eqref{eq:logderiv}:
\begin{equation}\label{eq:nonelementary}
\frac{Z^{\prime}_{\pi}(u)}{Z_{\pi}(u)} = R_1(u) + R_2(u)\ellipticK\left(\frac{4u}{1+3u^2}\right),
\end{equation}
where $R_1(u) = u^{-1}(1-u^2)^{-1}(3u^2-1)$ and $R_2(u) = \frac{2}{\pi}u^{-1}(1+3u^2)^{-1}(1-3u^2)$ are rational functions of $u$.  If $Z_{\pi}(u)$ was elementary, then the left hand side of~\eqref{eq:nonelementary} would be elementary, and so would $\ellipticK$. But $\ellipticK$ is non-elementary, a result of Liouville's.
\end{proof}

\section{Uniformization}\label{sec:uniformization}
In this section, theta functions are used to uniformize the elliptic integral in~\eqref{eq:logderiv}, and from that to analytically extend the zeta function.
For background on this material see~\cite{cox:agm} and \cite{borwein2}.  Another classic reference is~\cite[Chapter XXI]{whitakerwatson}.

%
\subsection{Theta functions and $k$}
Introduce $\tau \in \H$, where $\H$ is the upper half plane, and define the ``nome'' $q = e^{\pi i \tau} \in \D$, where $\D$ is the open unit disk. 
The theta functions are defined by series in powers of $q$, and also have infinite product expansions:
\begin{align}
\theta_2 &= \sum_{n=-\infty}^{\infty} q^{(n + 1/2)^2} = 2q^{1/4} \prod_{n=1}^\infty (1-q^{2n})(1+q^{2n})^2\\
\theta_3 &= \sum_{n=-\infty}^{\infty} q^{n^2} = \prod_{n=1}^\infty (1-q^{2n})(1+q^{2n-1})^2\\
\theta_4 &= \sum_{n=-\infty}^{\infty} (-1)^n q^{n^2} = \prod_{n=1}^\infty (1-q^{2n})(1-q^{2n-1})^2
\end{align}
They are commonly defined as functions of two variables, $\theta_*(z,q)$, but here only the specialization to $z=0$ is needed. 

The latter functions $\theta_3$ and $\theta_4$ are nonzero analytic functions of $q \in \D$.  However, the definition of
$\theta_2$, while quite standard, is misleading.  In fact, $\theta_2$ is not a continuous function of $q \in \D$,
due to the fractional powers in the defining series.
It should be considered as a function of $\tau$, with $q^\alpha$ interpreted as $e^{\pi i \tau \alpha}$.

Set
\[ k = \frac{\theta_2^2}{\theta_3^2}. \]

Like $\theta_2$, $k$ is an analytic function of $\tau \in \H$ but not continuous on $q \in \D$.  The natural variable for expressing $k$ and 
the zeta function turns out to be
\[ t = e^{i \pi \tau /2} \]
so that $q = t^2$.
The product formula gives $\theta_2^2(t) = 4t \prod_{n=1}^\infty (1-t^{4n})^2(1+t^{4n})^4$, so that both $\theta_2^2$ and $k$ are
analytic functions of $t \in \D$.

\begin{figure}
\begin{tabular}{ccc}
\includegraphics[width=.3\textwidth]{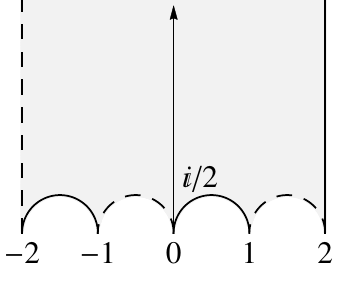}&
 \includegraphics[width=.3\textwidth]{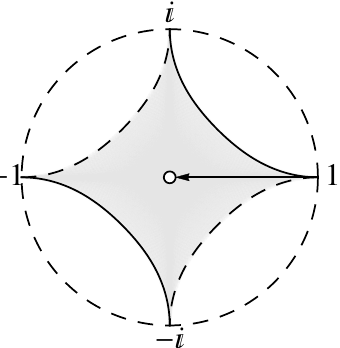}&
 \includegraphics[width=.3\textwidth]{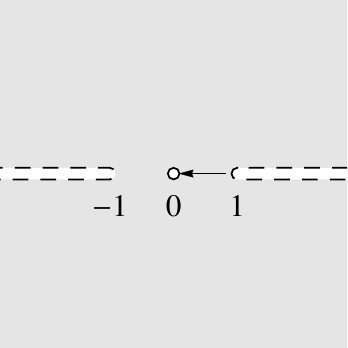}\\
$\mathcal{F} \subset \H$ & \pbox{1.5in}{$t(\mathcal{F}) \subset \D$ \\ $\mathcal{G}_0 := t(\mathcal{F})^\circ$} & $k(\mathcal{G}_0) \subset \C$
\end{tabular}
\caption{}
\label{fig:Fdomain}
\end{figure}

The map $\tau \to k$ takes all values of $k \in \C$ except for $0, \pm1$, and in fact gives the universal cover $\H \to \C - \{0,\pm 1\}$.
The deck transformations are an order two subgroup of the modular group $\Gamma(2)/\!\!\pm\!I$, acting by Mobius transformations.
A fundamental domain for this covering is $\mathcal{F}$ as shown in Figure~\ref{fig:Fdomain}.
See~\cite{cox:agm} for a discussion and proofs.

The map $\tau \to t=e^{i \pi \tau/2}$ takes $\mathcal{F}$ to the set $t(\mathcal{F})$ shown in Figure~\ref{fig:Fdomain}.
Beware that the boundary curves of $t(\mathcal{F})$ are not circular arcs.  Let $\mathcal{G}_0$ be the interior
of $t(\mathcal{F})$.  Then the map $t \to k$ takes $\mathcal{G}_0$ homeomorphically to $\C - (-\infty,-1] \cup 0 \cup [1, \infty)$.

The maps $\tau \to t \to k$ also give homeomorphisms
\[ \tau \in \text{imaginary axis of\ } \H \longleftrightarrow t \in (0,1) \longleftrightarrow k \in (0,1) \]
shown as solid arrows in Figure~\ref{fig:Fdomain}.

There is a subtle point here that will come up repeatedly.  The map $\H \to \C - \{0,\pm 1\}$ is the universal covering map, and it
factors through $t \in \D_0 := \D - 0$.  The map $t \to k$ has a removable singularity at $0$, but removing it means that $t \to k$ is no
longer a covering map, since every $k$ is the image of infinitely many $t$, except for $k = 0$.
Though it adds complication, the 0 point is important and so it is worth considering how it behaves.  Notationally, sets which are
missing the zero point will have subscript $0$'s.  Here, $\mathcal{G} := \mathcal{G}_0 \cup 0$, and the map $t \to k$ sends
$\mathcal{G}$ homeomorphically to $\C - (-\infty, -1] \cup [1,\infty)$.

\subsection{The uniformizing surface}\label{sec:uniform}
To uniformize~\eqref{eq:logderiv}, we need a surface on which both $k$ and $u$ are defined.  As a first step, set
\[ S_1 = \Bigl\{ (u,t)\ \Bigl|\ \frac{4u}{1+3u^2} = k= \frac{\theta_2^2(t)}{\theta_3^2(t)} \Bigr\} \subset \C \times \D. \]
with projection maps $\Pi_u : S_1 \to \C$ and $\Pi_t : S_1 \to \D$.

The map $u \to k$ is a double cover, singular when $u = \pm i/\sqrt{3}$ and branched at $k = \pm 2/\sqrt{3}$.  Let
\[ K = \bigl\{ t \in \D\ \bigl|\ k(t) = \pm2/\sqrt{3}\bigr\} \]
The set $K$ is discrete, since it is the lift of $k = \pm 2/\sqrt{3}$ by the covering map $t \to k$,
and each point of $K$ corresponds to a branch point of $S_1$.
\begin{Lemma}
The set $S = \bigl\{ (u,t) \in S_1\ \bigl|\ k(t) \neq \pm 2/\sqrt{3}\bigr\}$
is a Riemann surface.
\end{Lemma}
\begin{proof}
Let $(u,t) \in S$, and suppose $k \neq 0$.  Then
$u = \frac{2 + \sqrt{4-3 k^2}}{3k}$ or $u = \frac{2 - \sqrt{4-3 k^2}}{3k}$.
Since $4-3k^2 \neq 0$, one or the other of these will hold in a neighborhood of $(u,t)$, and $\Pi_t$ will be a homeomorphism
from this neighborhood to an open subset of $\D$.

The $k = 0$ case occurs only for the point $(0,0) \in S$.
Again, $\Pi_t$ is a local homeomorphism, this time because
$u = \frac{2 - \sqrt{4-3 k^2}}{3k}$ in a neighborhood of $(0,0)$ and the singularity at $k = 0$ is removable.
\end{proof}

The largest domain of definition for the zeta function will turn out to be $S$.  In the $u$ plane, this corresponds to
$Y := \Pi_u(S)$.  Since $k(t) \neq \pm 1$, $Y$ does not contain $u = \pm 1, \pm \frac{1}{3}$.  The branch points
correspond to $u = \pm \frac{1}{\sqrt{3}}$, and $k(u)$ is singular when $u = \pm \frac{i}{\sqrt{3}}$.  Thus
\[ Y = \Pi_u(S) =  \C - \{\pm \frac{1}{3}, \pm \frac{1}{\sqrt{3}}, \pm \frac{i}{\sqrt{3}}, \pm 1\}. \quad\text{(See Figure~\ref{fig:Y})} \]
The corresponding projection for $t$ is $\Pi_t(S) = \D - K$.

Every point of $Y$ is covered infinitely many times by $S$ with the exception of $u = 0$, covered once.  Every point of $\D - K$ is covered twice by $S$, again with the exception of $t = 0$.  It is then natural to define $S_0 = S - (0,0)$ and $Y_0 = Y - 0$, giving rise
to a commutative diagram of covering maps
\[
\begin{CD}
\sigma = (u,t) \in S_0 @>{\Pi_t}>> t \in \D_0- K \\
@V{\Pi_u}VV              @VVV\\
u \in Y_0           @>>> k \in \C - \{0, \pm 1, \pm 2/\sqrt{3}\}
\end{CD}
\]
where the horizontal maps are degree 2 and the vertical maps have infinite degree.

\begin{figure}
\includegraphics[width=2in]{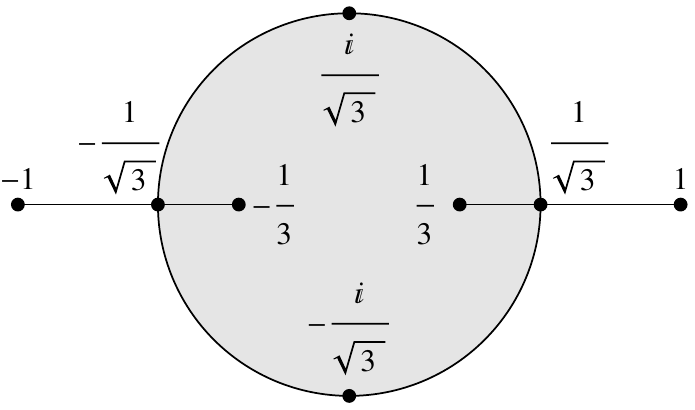}
\caption{$Y$ is the complement of the indicated points. $D$ is the segments and circle.  $\Omega$ is the finite component of $\C - D$, shaded.}
\label{fig:Y}
\end{figure}

The zeta function is defined on the set $\Omega$, shown in Figure~\ref{fig:Y}.  The map $u \to k$ takes $\Omega$ homeomorphically
to $\C - (-\infty,-1]\cup[1,\infty)$, which in turn is homeomorphic to $\mathcal{G}$.
Lifting to $S$, let
\[ \tilde{\Omega} = \{ (u,t) \in S\ |\ \Pi_u(u) \in \Omega \text{\ and\ } \Pi_t(t) \in \mathcal{G} \}; 
   \quad \tilde{\Omega}_0 = \tilde{\Omega} - (0,0)\]
The projections $\Pi_u, \Pi_t$ map $\tilde{\Omega}$ homeomorphically to $\Omega$ and $\mathcal{G}$ respectively.

\subsection{The extended zeta function}
This section analytically extends the zeta function, and proves the main result of this article.
\begin{Thm}\label{thm:extends}
The zeta function $Z_\pi(u)$ extends analytically to a multivalued function on $Y$ in the following sense:
There is a Riemann surface $S$, a holomorphic map $\Pi : S \to Y$
and a holomorphic function $\tilde{Z}_\pi$ on $S$ with the property that $Z_\pi(\Pi(\sigma)) = \tilde{Z}_\pi(\sigma)$ for
$\sigma \in \tilde{\Omega}$,
where $\tilde{\Omega}$ is a particular lift of $\Omega$ to $S$.

Moreover, $\Pi \bigr|_{S_0} : S_0 \to Y_0$ is a covering map of infinite degree.
\end{Thm}

With $\Pi = \Pi_u$, the existence and properties of $S, S_0, Y, Y_0$, and $\tilde{\Omega}$ were all established in the previous section.
It remains to define $\tilde{Z}_\pi$.  Theorem~\ref{thm:logderiv} is the starting point.

\begin{Lemma}\label{lem:exact}
With $k = k(t) = \theta_2^2/\theta_3^2$, there is $F$ analytic on $\D$ with $F(0) = 0$ and
\[ -2\pi^{-1}\ellipticK(k)\frac{dk}{k} = d(F(t) - \log t) \]
\end{Lemma}
\begin{proof}
From~\cite[Theorem 2.1 and (2.3.11)]{borwein2}
\begin{align}
  \ellipticK(k) &= \frac{\pi}{2} \theta_3^2\\
  \frac{d\theta_2}{\theta_2} - \frac{d\theta_3}{\theta_3} &= \frac{i \pi}{4}\theta_4^4 d\tau
\end{align}
Then, using $dt/t = \frac{i\pi}{2}d\tau$,
\[ \frac{dk}{k} = 2\frac{d\theta_2}{\theta_2}- 2\frac{d\theta_3}{\theta_3} = \theta_4^4\frac{dt}{t} \]
and
\[ -2\pi^{-1}\ellipticK(k)\frac{dk}{k} = -\theta_3^2\theta_4^4 \frac{dt}{t}. \]

Put $f(t) = \frac{1}{t}\bigl(1 - \theta_3^2\theta_4^4\bigr), f(0) = 0$.  Then $f$ is analytic on the unit disk, and has a unique primitive $F(t)$ with $F(0) = 0$ and $dF = f dt$.
It follows that
\[ d(F - \log t) = f dt - \frac{dt}{t} = -\theta_3^2\theta_4^4 \frac{dt}{t} \]
\end{proof}

\begin{Lemma}\label{lem:det}
For $(u,t) \in \tilde{\Omega}_0$,
\[ u^{-1}\detDeltau = t^{-1}e^{F(t)} \]
where $F$ is the analytic function from Lemma~\ref{lem:exact}.
\end{Lemma}
\begin{proof}
Equation~\eqref{eq:logderiv} becomes
\[
d(\log \detDeltau) - \frac{du}{u} = 2\pi^{-1}\ellipticK(k)\frac{dk}{k}
\]
so that
\[
d(\log u^{-1}\detDeltau) = d(F(t) - \log(t))
\]
and therefore
\[ u^{-1}\detDeltau = Ct^{-1}e^{F(t)} \]
for some $C \neq 0$.
As $(u,t) \to (0,0)$, $k = 4u/(1+3u^2) \sim 4u$, and $k = \theta_2^2/\theta_3^2 \sim 4t$ which shows that $u/t \to 1$.
Also, $\detDeltau \to 1$ and $F(t) \to 0$, so that $C = 1$.
\end{proof}

\begin{Def}\label{def:zeta}
For $\sigma = (u,t) \in S_0$, define
\[ \tilde{Z}_\pi(\sigma) = \frac{t e^{-F(t)}}{u(1-u^2)}. \]
As observed in Lemma~\ref{lem:det}, $t/u \to 1$ as $\sigma \to (0,0)$, so the singularity at $(0,0)$ is removable and
$\tilde{Z}_\pi$ extends to a holomorphic function on $S$ with
\[ \tilde{Z}_\pi(0,0) = 1. \]
\end{Def}

For $\sigma \in \tilde{\Omega}_0$, apply $\eqref{eq:zetadef}$ and Lemma~\ref{lem:det} to get
\[ Z_\pi(\Pi(\sigma)) = Z_\pi(u) = \frac{1}{(1-u^2)\detDeltau} = \frac{t e^{-F(t)}}{u(1-u^2)} = \tilde{Z}_\pi(\sigma). \]

This completes the proof of Theorem~\ref{thm:extends}.

\section{Applications of the zeta function formula}
\subsection{Graphing}

\begin{figure}
\begin{tabular}{cc}
  \includegraphics[width=.5\textwidth]{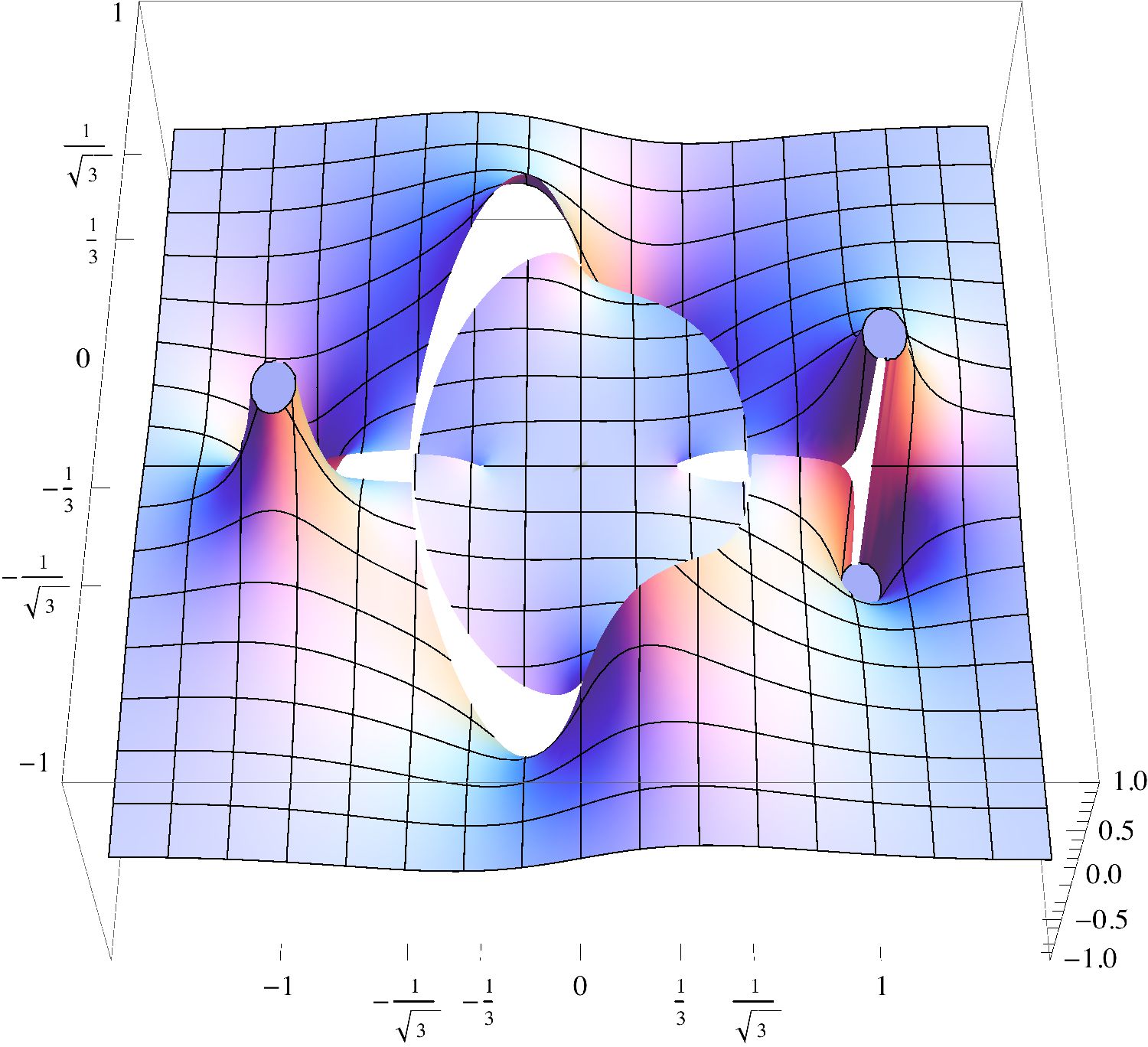} &
  \includegraphics[width=.5\textwidth]{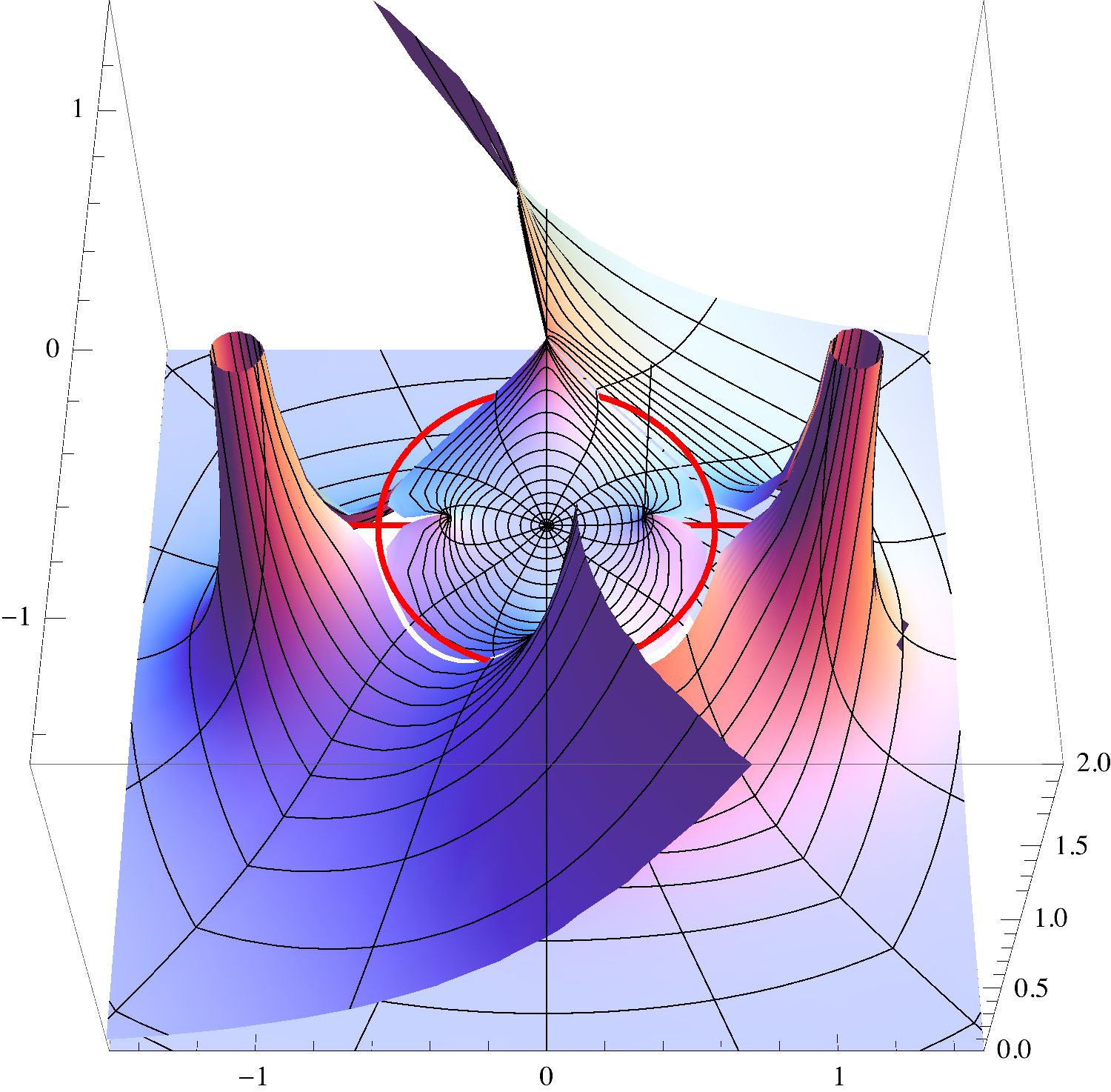} \\
a. $\Im Z_\pi(u)$ & b. $\abs{Z_\pi(u)}$
\end{tabular}
\caption{Graphs of the zeta function of the grid.}\label{fig:graphs}
\end{figure}

The zeta function is given in Definition~\ref{def:zeta} as $\tilde{Z}_\pi(\sigma) = \frac{t e^{-F(t)}}{u(1-u^2)}$.
There are two issues involved when applying this formula to explicitly compute $\tilde{Z}_\pi$.

First, there is the function $F$, defined by
$dF = \frac{1}{t}\bigl(1 - \theta_3^2\theta_4^4\bigr)dt$.  This $F$ can be effectively computed by expanding the theta
functions as series in $t$ and then integrating term by term, since both sides are analytic on the entire disk $\D$.

The second issue is that
the zeta function is parameterized by $\sigma \in S$, so that its graph as a multivalued function of $u$
is given by $(\Pi_u(\sigma),\tilde{Z}_\pi(\sigma)) \subset \C \times \C$.  In theory, one could pass to the universal
cover of $S$ and parameterize all of $S$ as a function of a variable in $\H$ or $\D$, but this seems intractable. 
Locally, however, $t$ and $u$ are implicit functions of each other, so one can choose either as a free variable.

Consider using $u$ as the parameter.
It follows from ~\cite[Theorem 2.3]{borwein2} that one can compute $t$ from $u$ as
\begin{equation}\label{eq:tfromu}
   u \to k = \frac{4u}{1+3u^2} \to  \tau = \frac{i\,\ellipticK(\sqrt{1-k^2})}{\ellipticK(k)} \to t = e^{i\pi \tau/2}.
\end{equation}
The square root $\sqrt{1-k^2}$ is not an issue, since $\ellipticK$ is really a function of $k^2$. In fact, Mathematica
defines ${\tt EllipticK}$ so that $\ellipticK(k) = {\tt EllipticK[}k^2{\tt]}$.

Figure~\ref{fig:graphs}a shows the imaginary part of the zeta function
\footnote{Mathematica code for computing the grid zeta function is available on the author's web page.},
graphed in this way as a function of $u$.
It has branch cuts along the entire set $D$, because of branch cuts in $\ellipticK$.  To heal these cuts, one needs to
choose different values of the multivalued elliptic integral $\ellipticK$.
From Cox~\cite{cox:agm}, there is a free group of Mobius transformations that acts on $\tau$ to produce all possible values of $\tau$
for a given $k$.
This group consists of integer matrices $\begin{pmatrix}a & b\\c & d\end{pmatrix}$ of determinant one which are the identity modulo 2 and
with $b \equiv 0 \pmod{4}$.  It is free on three generators (as it is index two in the modular group $\Gamma(2)/\pm \!I$).
Thus, one can produce any value of $\tilde{Z}_\pi$ by multiplying $\tau$ by an element of this group.

It is somewhat simpler to begin with $t$ as the parameter.
Then
\[ t \to  q=t^2 \to k = \frac{\theta_2^2(q)}{\theta_3^2(q)} \to u_\pm =   \frac{2 \pm \sqrt{4-3 k^2}}{3k}. \]
As $t$ ranges over the unit disk $\D$ and both sign choices are used in $u_\pm$, all points on the graph of $\tilde{Z}_\pi$ are attained.
Plotting for both values $u_\pm$ at once then produces a smooth graph of $Z_\pi(u)$, although in practice the
jumps when
$u_+$ and $u_-$ switch cause problems for computer graphing algorithms.

Figure~\ref{fig:graphs}b shows two sheets of $\abs{Z_\pi(u)}$, plotted parametrically as a function of $t$.  The set $D$ is also shown, in red.

\subsection{The Functional Equation}
For a finite graph $X$, the Ihara zeta function $\zeta(u)$ is a polynomial.
If $X$ is $q+1$-regular with $v$ vertices and $e$ edges, $\zeta(u)$ satisfies a functional equation relating $u$ with $\frac{1}{qu}$.
This equation takes the form
\begin{equation}\label{eq:dual}
   \zeta(\frac{1}{qu}) = q^{2e-v} u^{2e} \left(\frac{1-u^2}{q^2u^2-1}\right) \zeta(u).
\end{equation}

The situation for an infinite $\pi$-periodic graph $X$ is complicated by the fact that the zeta function $Z_\pi$ for $X$ is only defined in some neighborhood of $0$.  There is a set $D_X$ of potential singularities which separates the $u$ plane into two regions.
A~priori, $Z_\pi$ is defined only on the finite region $\Omega_X$, exactly as in Figure~\ref{fig:Y} but with $q$ instead of $3$.
The relation $u \leftrightarrow \frac{1}{qu}$ exchanges $\Omega_X$ with the infinite component of $\C - D_X$, so
until $Z_\pi$ is extended, only one side of an equation like~\eqref{eq:dual} makes sense for any given $u$.

One might be tempted to extend $Z_\pi$ by using a functional equation similar to \eqref{eq:dual} as a definition,
but this can produce functions which fail to match at all along the entire separating set $D_X$. When $X$ is the infinite path
as in Example~\ref{ex:path}, the zeta function is $1$ inside the unit circle while duality gives
$u^2$ outside the circle.  The situation with Cayley graphs of free groups in Example~\ref{ex:tree} is similar.

On the other hand, Guido and Isola show in~\cite{gi:duals} that if the spectrum of the adjacency operator for $X$ has certain gaps,
then extending $Z_\pi$ using duality will match an analytic continuation of $Z_\pi$ across corresponding gaps in $D_X$.

In the case of the infinite grid, the spectrum of the adjacency operator $\adj$ has no gaps -- it is the interval $[-4,4]$. 
Applying the functional equation to try and extend the zeta function for the infinite grid results in the
function shown in Figure~\ref{fig:graphs}a, which fails to be analytic at any point of $D$.
Using the multivalued analytic continuation does produce a functional equation analogous to~\eqref{eq:dual}:

\begin{Thm}\label{thm:duality}
The surface $S_0$ from section~\ref{sec:uniform} has an involution $\iota: (u,t) \to (\frac{1}{3u},t)$, and the zeta function
$\tilde{Z}_\pi : S_0 \to \C$ satisfies the functional equation
\[ \tilde{Z}_\pi(\iota(\sigma)) = 3^3 u^4 \left(\frac{1-u^2}{3^2 u^2 - 1}\right)\tilde{Z}_\pi(\sigma) \]
for all $\sigma = (u,t) \in S_0$.
\end{Thm}
\begin{Remark*}
This functional equation is a direct analog of~\eqref{eq:dual} with $q = 3$, $v = 1$, and $e = 2$.  In the case of periodic graphs,
the natural values of $v$ and $e$ come from the quotient graph $X/\pi$.
\end{Remark*}
\begin{proof}
Observe that if $(u,t) \in S_0$ then $(\frac{1}{3u},t) \in S_0$, since
\[ \frac{4u}{1+3u^2} = k = \frac{4 \left(\frac{1}{3u}\right)}{1 + 3\left(\frac{1}{3u}\right)^2}. \]
Also $(u,t)$ and $(\frac{1}{3u},t)$ are distinct, since $u = \frac{1}{3u}$ implies $u = \pm 1/\sqrt{3} \notin \Pi_u(S_0).$
Thus $\iota$ is an involution.

From Definition~\ref{def:zeta}
\begin{align}
    \tilde{Z}_\pi(\iota(\sigma))
   &= \frac{t e^{-F(t)}}{\frac{1}{3u}(1-\frac{1}{(3u)^2})}\\
   &= 3^3 u^4 \left(\frac{1-u^2}{3^2 u^2 - 1}\right) \frac{t e^{-F(t)}}{u(1-u^2)}\\
   &=  3^3 u^4 \left(\frac{1-u^2}{3^2 u^2 - 1}\right) \tilde{Z}_\pi(\sigma)
\end{align}
\end{proof}

\subsection{Limits of finite graphs}
There are many finite approximation results in the study of $\Ltwo$ invariants, most of which originiate from topological applications.
These results begin with an infinite object $X$ (cell complex, graph, discrete group, ...) and a directed set of finite 
objects $\{X_\alpha\}$ that
approximate $X$, and then show that the spectra of operators on $X_\alpha$ approach the spectrum of a related operator on $X$.
Chapter 13 of~\cite{luck:book} gives a very general treatment.

Two of these approximation results have been worked out for zeta functions of periodic graphs $(X,\pi)$ with $X \to X/\pi$ 
a finite regular covering.  In~\cite{cms:zeta2}, there are finite index subgroups $\pi_\alpha \normal\ \pi$ with
$\cap_\alpha \pi_\alpha = \{e\}$, and the approximating graphs are $X_\alpha = X/\pi_\alpha$.
In~\cite{gil:amenable}, $\pi$ must be an amenable group, and the $X_\alpha$ form a F{\o}lner approximation to $X$.
Both of these approximation
results apply in full generality to the infinite grid.  Here, we simply call attention to a particularly interesting special case.

Let $P_n$ and $C_n$ be the path and cycle graphs on $n$ vertices.  The graphs $P_n \times P_m$ and $C_n \times C_m$ are sometimes
called \define{square grid graphs} and \define{torus graphs}, respectively.

\begin{Thm}
Let $X_\alpha$ be $P_{n(\alpha)} \times P_{m(\alpha)}$ or $C_{n(\alpha)} \times C_{m(\alpha)}$, a sequence of square grid or torus graphs,
with $n(\alpha), m(\alpha) \to \infty$.  Denote the Ihara zeta function of $X_\alpha$ by $\zeta_\alpha(u)$.  Then
for $u$ in a neighborhood of $0$,
\[
  \lim_{\alpha \to \infty} \zeta_\alpha(u)^{\frac{1}{n(\alpha)m(\alpha)}} = Z_\pi(u) = \frac{t e^{-F(t)}}{u(1-u^2)}
\]
where $F$ satisfies $F(0)=0$, $dF/dt = \frac{1}{t}\bigl(1 - \theta_3^2(t^2)\theta_4^4(t^2)\bigr)$ and $t$ is given as a function of $u$
by~\eqref{eq:tfromu}.  In the case case of torus graphs, this holds for $u \in \Omega$.  In the square grid case, it holds (at least) for
$\abs{u} < (4 + \sqrt{22})^{-1} \approx 0.115$.
\end{Thm}
\begin{proof}
The grid graphs $P_{n(\alpha)} \times P_{m(\alpha)}$ form a F{\o}lner approximation to the infinite grid, so the approximation
theorem of~\cite{gil:amenable} applies.  The torus graphs $C_{n(\alpha)} \times C_{m(\alpha)}$ are quotients of the infinite
grid by $n(\alpha)\Z \times m(\alpha)\Z$ and the approximation theorem of~\cite{cms:zeta2} applies.
\end{proof}
%
%
\section{A Combinatorial Approach}\label{sec:combinatoric}
Here, we compute $\log \detDeltau = \Tr_{\pi} \Log \Delta_u$ by calculating the power series for $\Log \Delta_u$.  Begin with $\log(1 - x) = -x - \frac{1}{2}x^2 - \frac{1}{3}x^3 - \dotsb$.
\begin{align}
  \Tr_{\pi} \Log \Delta_u &= \Tr_{\pi} \Log (1 - (-3u^2 + u(a + a^{-1} + b + b^{-1})) \\
   &= - \Tr_{\pi} \sum_{n=1}^{\infty}\frac{1}{n}(-3u^2 + u(a + a^{-1} + b + b^{-1}))^n \\
   &= -  \Tr_{\pi} \sum_{n=1}^{\infty}\frac{u^n}{n}(-3u + (a + a^{-1} + b + b^{-1}))^n \\
   &= - \sum_{n=1}^{\infty}\frac{u^n}{n} \sum_{i=0}^n \binom{n}{i} (-3u)^{n-i} \Tr_{\pi} (a + a^{-1} + b + b^{-1})^{i}\\
   &= - \sum_{n=1}^{\infty}\frac{u^n}{n} \sum_{k=0}^{\floor{n/2}}
            \binom{n}{2k} (-3u)^{n-2k} \Tr_{\pi} (a + a^{-1} + b + b^{-1})^{2k}
\end{align}
Where the final step above is to drop the odd powers, which vanish because $\Tr_{\pi}g = 0$ unless $g = e$ is the identity. Now,
\begin{align}
  \Tr_{\pi} (a + a^{-1} + b + b^{-1})^{2k}
     &= \Tr_{\pi}\sum_{j=0}^{k} \binom{2k}{2j}(a + a^{-1})^{2j}(b + b^{-1})^{2k-2j}\\
     &= \sum_{j=0}^{k} \binom{2k}{2j}\binom{2j}{j}\binom{2k-2j}{k-j}\\
     &= \sum_{j=0}^{k} \frac{(2k)!}{j!^2(k-j)!^2} \\
     &= \binom{2k}{k} \sum_{j=0}^{k}\binom{k}{j}^2\\
     &= \binom{2k}{k}^2. 
\end{align}
Then
\begin{align}
  \Tr_{\pi} \Log \Delta_u
     &= - \sum_{n=1}^{\infty}\frac{u^n}{n} \sum_{k=0}^{\floor{n/2}}
            \binom{n}{2k} (-3u)^{n-2k} \binom{2k}{k}^2\\
     &= - \sum_{n=1}^{\infty} \sum_{k=0}^{\floor{n/2}}
            \frac{(-3)^{n-2k}}{n}\binom{n}{2k} \binom{2k}{k}^2 u^{2n-2k}
\end{align}
Now put $M = n-k$, so
\begin{align}
  \Tr_{\pi} \Log \Delta_u
    &= - \sum_{M=1}^{\infty} \left[ \sum_{k=0}^{M}
            \frac{(-3)^{M-k}}{M+k}\binom{M+k}{2k} \binom{2k}{k}^2\right] u^{2M}\\
    &= u^2-\frac{3 u^4}{2}-\frac{11 u^6}{3}-\frac{107 u^8}{4}-\frac{759u^{10}}{5}-\frac{6039 u^{12}}{6}-\dotsb \label{eq:trlogseries}
\end{align}
Mathematica claims the coefficient of $u^{2M}$ can be represented in terms of a hyper\-geo\-metric function as
 $-(-3)^M \, _3F_2\left(\frac{1}{2},-M,M;1,1;\frac{4}{3}\right)/M$,
but that seems to be of no particular help in moving towards a general term in the series representation of $Z_\pi$.  However, one can
exponentiate the beginning of the series~\eqref{eq:trlogseries} and determine:
\begin{align}
  \detDeltau &= \exp \Tr_{\pi} \Log \Delta_u\\
   \label{eq:detseries}
    &= 1+u^2-u^4-5 u^6-30 u^8-174 u^{10}-1120 u^{12} + \dotsb
\end{align}   
Finally,
\begin{align}
Z_{\pi}(u) &= \frac{ \exp\left( \sum_{M=1}^{\infty} \left[ \sum_{k=0}^{M}
            \frac{(-3)^{M-k}}{M+k}\binom{M+k}{2k} \binom{2k}{k}^2\right] u^{2M}\right) } {1-u^2}\\
\begin{split}\label{eq:zseries}
           = 1+2 u^4+4 u^6+29 u^8+160 u^{10}+1070 u^{12}+7192 u^{14}+ \\
            + 50688u^{16}+365376 u^{18}+2695122 u^{20}+ \dotsb 
\end{split}
\end{align}

%
%
\bibliographystyle{plain}
\bibliography{zetagrid}

\end{document}